\documentclass[12pt]{amsart}
\usepackage[margin=1in]{geometry}
\usepackage{amsthm}
\usepackage{amssymb}   % AMS fonts
\usepackage{mathtools}
\usepackage{url}

\numberwithin{equation}{section}
\numberwithin{figure}{section}

\theoremstyle{plain}
 \newtheorem{thm}{Theorem}[section]
 \newtheorem{lem}[thm]{Lemma}
 \newtheorem{prp}[thm]{Proposition}
 \newtheorem{cor}[thm]{Corollary}
\theoremstyle{definition}
 
\theoremstyle{remark}
 \newtheorem{rem}[thm]{Remark}
 \newtheorem{exm}[thm]{Example}
 \newtheorem{prb}[thm]{Problem}

\newcommand{\byeqn}[1]{\overset{\eqref{#1}}{=}}

\newcommand{\id}[1]{\mathrm{id}_{#1}}
\newcommand{\ldv}{\backslash}
\newcommand{\rdv}{/}
\newcommand{\inv}{^{-1}}
\newcommand{\linv}{^{\ell}}
\newcommand{\rinv}{^r}
\newcommand{\lsec}[1]{L_{(#1)}}
\newcommand{\rsec}[1]{R_{(#1)}}
\newcommand{\mlt}[1]{\mathrm{Mlt}(#1)}
\newcommand{\lmlt}[1]{\mathrm{Mlt}_{\ell}(#1)}
\newcommand{\rmlt}[1]{\mathrm{Mlt}_r(#1)}
\newcommand{\inn}[1]{\mathrm{Inn}(#1)}
\newcommand{\linn}[1]{\mathrm{Inn}_{\ell}(#1)}
\newcommand{\rinn}[1]{\mathrm{Inn}_r(#1)}
\newcommand{\nuc}[1]{\mathrm{Nuc}(#1)}
\newcommand{\lnuc}[1]{\mathrm{Nuc}_{\ell}(#1)}
\newcommand{\mnuc}[1]{\mathrm{Nuc}_m(#1)}
\newcommand{\rnuc}[1]{\mathrm{Nuc}_r(#1)}
\newcommand{\lmnuc}[1]{\mathrm{Nuc}_{\ell,m}(#1)}
\newcommand{\lrnuc}[1]{\mathrm{Nuc}_{\ell,r}(#1)}
\newcommand{\rmnuc}[1]{\mathrm{Nuc}_{r,m}(#1)}
\newcommand{\atp}[1]{\mathrm{Atp}(#1)}

\title{Loops with squares in two nuclei}

\author{Michael Kinyon}
\address[Kinyon]{Department of Mathematics, University of Denver, Denver, CO 80208 USA}
\email{michael.kinyon@du.edu}

\author{J. D. Phillips}
\address[Phillips]{Department of Mathematics and Computer Science, Northern Michigan University, Marquette, MI 49855 USA}
\email{jophilli@nmu.edu}

\date{\today}

\begin{document}

\begin{abstract}
Although little can be gleaned about a loop with the property that its squares are, say, left nuclear ($xx\cdot yz = (xx\cdot y)z$), if its squares are also, say, middle nuclear ($(x\cdot yy)z = x(yy\cdot z)$), then the loop exhibits more structure than one might initially guess. Loops with squares in (at least) two nuclei include many well known classes of loops, such as C loops and extra loops, and not so well known classes such as left C loops. In any loop with, say, left and middle nuclear squares, the intersection of the left and middle nuclei is a normal subloop; hence such a loop is simple if and only if it is a group or a simple unipotent loop. Loops in which squaring is a centralizing endomorphism have even more structure; they are power-associative, and a torsion loop in that class is a direct product of a loop of $2$-elements and a loop of elements of odd order.
\end{abstract}

\maketitle

\section{Introduction}
\label{Sec:intro}

A \emph{loop} $(Q,\cdot,e)$ is a set $Q$ with a binary operation $\cdot$ such that $e\cdot x = x = x\cdot e$ for all $x\in Q$, and for each $a\in Q$, the mappings $L_a\colon Q\to Q; x\mapsto ax$ and $R_a\colon Q\to Q; x\mapsto xa$ are bijections. Basic references for loop theory are \cite{Bel}, \cite{Br}, \cite{Pf}. We adopt a standard notational convention for nonassociative
structures to avoid excessive parentheses: juxtaposition has priority over the displayed binary operation $\cdot$ in terms to be multiplied. For example, the identity $x(y\cdot yz) = (x\cdot yy)z$ is shorthand for $x \cdot (y \cdot (y \cdot z)) = (x \cdot (y \cdot y))\cdot z$.

A universally quantified identity in $3$ distinct variables is said to be of \emph{Bol-Moufang type} if (i) all three variables occur on both sides of the equal sign in the same order, and (ii) exactly one of the variables appears twice on both sides. For example, the identity above is of Bol-Moufang type. Identities of Bol-Moufang type were first studied by Fenyves \cite{Fe} who sorted out the loop varieties (in the universal algebra sense) they define; this work was later refined and completed in \cite{PVII}.

There are 60 identities of Bol-Moufang type, and it turns out they define 14 distinct varieties of loops. Six of those varieties have been investigated quite thoroughly---groups, extra loops, Moufang loops, left Bol loops, right Bol loops, and C loops.

Another 6 of the 14 Bol-Moufang varieties of loops have so little structure that not much can be said about them individually: flexible loops, left alternative loops, right alternative loops and the following:
\begin{align}
	xx\cdot yz = (xx\cdot y)z & & \text{\emph{left nuclear squares}} \tag{LNS}\label{Eqn:LNS} \\
	(x\cdot yy)z = x(yy\cdot z) & & \text{\emph{middle nuclear squares}} \tag{MNS}\label{Eqn:MNS} \\
	xy\cdot zz = x(y\cdot zz) & & \text{\emph{right nuclear squares}}. \tag{RNS}\label{Eqn:RNS}
\end{align}

Later we will also have occasion to discuss the property $xx\cdot y = y\cdot xx$. We will refer to this as \emph{commuting squares}.

The remaining two Bol-Moufang loop varieties have interesting structure but, to our knowledge, have not been studied in much detail. \emph{Left C loops} are defined by any one of four equivalent Bol-Moufang identities \cite{Fe,PVII}: 
\begin{align*}
x(y \cdot yz) = (x \cdot yy)z, & & xx \cdot yz = (x \cdot xy)z, \\
x(x \cdot yz) = (x \cdot xy)z, & & x(x \cdot yz) = (xx \cdot y)z.
\end{align*}
\emph{Right C loops} are defined using the mirrors of these four identities, that is, a loop is right C if and only if its opposite loop is left C. Thus an investigation of left C loops suffices to understand right C loops. C loops themselves are defined by a Bol-Moufang identity of their own, namely $(xy\cdot y)z = x(y\cdot yz)$, but can be characterized as loops that are both left C and right C.

\begin{rem}
Fenyves \cite{Fe} originally dubbed left C loops as LC loops. Fenyves' intention with his use of the letter C is debatable; it has been suggested that it was short for ``central,'' although he certainly never wrote that explicitly. However, there can be little doubt that the L in LC stands for ``left''. We prefer ``left C'' over ``LC'' for aesthetic reasons but also for a practical one: LC is too easy to confuse with LCC, which is the standard acronym for the variety of \emph{left conjugacy closed} loops, a more highly structured and well studied variety.
\end{rem}

Our original motivation for this paper was a detailed study of left C loops. They have an important property \cite{Fe}. 

\begin{prp}
	Every left C loop has both left nuclear and middle nuclear squares. 
\end{prp}

\noindent Dually, right C loops have both middle nuclear and right nuclear squares. C loops, and hence extra loops, have squares in all three nuclei.

In the course of our investigations, we found that the structure of left C loops is largely determined by the property of the proposition, and so we shifted from our original task to the more general study of the pairwise intersections of the left nuclear square, middle nuclear square, and right nuclear square varieties.  

After a review of basic loop theory in {\S}\ref{Sec:basics}, we discuss principal loop isostrophes in {\S}\ref{Sec:isostrophes} which will be our main tool for transferring results between the aforementioned intersection varieties. In {\S}\ref{Sec:two_nuc}, we turn to our main results. In Theorem \ref{Thm:LMNuc_normal}, we show that if $Q$ is a loop with left and middle nuclear squares, then the intersection of the left and middle nuclei is a normal subloop. This was already known for left C loops \cite{DK}. Using principal isostrophes, we then prove the corresponding result for loops with left and right nuclear squares (Theorem \ref{Thm:LRNuc_normal}).

In {\S}\ref{Sec:AIP}, we study the subvariety of loops in which squaring is a centralizing endomorphism, that is, an endomorphism taking its values in the loop center. Our main result of the section, Theorem \ref{Thm:AIP_SqEndo}, characterizes such loops as those with squares in two nuclei and endomorphic squaring or, equivalently, with squares in two nuclei and the automorphic inverse property. These loops also turn out to be power-associative (Lemma \ref{Lem:CS_PA}).

The main result of {\S}\ref{Sec:decomp} is Theorem \ref{Thm:ExO}, a decomposition theorem showing that a torsion loop in the variety of loops with centralizing endomorphic squaring is a direct product of a loop of $2$-elements and a loop in which every element has odd order. This is the analog of similar decomposition results for certain commutative diassociative loops \cite{commdi}, commutative automorphic loops \cite{commaut}, and Bruck loops \cite{AKP,Bau}.

We wrap up the paper with {\S}\ref{Sec:leftC}, a discussion of the implications of our results for left C loops.

\section{Basics}
\label{Sec:basics}

In this section we review some of the basics of loop theory we will need in what follows. For uncited assertions, we refer the reader to the standard references \cite{Bel}, \cite{Br}, \cite{Pf}.

For elements $a,b$ of a loop $Q$, let $a\ldv b$ and $b\rdv a$ denote, respectively, the unique solutions $x$ and $y$ to the equations $ax=b$
and $ya=b$. This introduces the \emph{left} and \emph{right division} operations $\ldv$ and $\rdv$, which are easily seen to satisfy the
identities
\[
x\cdot x\ldv y = y = y\rdv x\cdot x\qquad\text{and}\qquad x\ldv xy = y = yx\rdv x\,.
\]
Here we use the standard notation convention that juxtaposition of the multiplication binds more tightly than the divisions, and the divisions,
in turn, bind more tightly than the explicit multiplication operation.

For all $x\in Q$, we abbreviate the \emph{left inverse} $e\rdv x$ and the \emph{right inverse} $x\ldv e$ by $x\linv$ and $x\rinv$, respectively. Thus $x\linv x = e$ and $xx\rinv = e$. We denote the corresponding permutations by $\lambda\colon Q\to Q; x\mapsto x\linv$ and
$\rho\colon Q\to Q; x\mapsto x\rinv$. For those $x$ satisfying $x\linv = x\rinv$, we denote the common value by $x\inv$, that is,
$x^{-1}$ is the (unique) two-sided inverse.

Note that in this paper, permutations act on the left of their arguments. The loop theory literature is not consistent with respect to this convention, and the present authors include themselves in that regard.

We have already noted that for each element $a$ of a loop $Q$, the \emph{left} and \emph{right translation} maps $L_a\colon Q\to Q;x \mapsto ax$ and $R_a\colon Q\to Q;x\mapsto xa$ are permutations (bijections) of $Q$. For a subloop $S$ of $Q$, let $\lsec{S}\coloneqq \{ L_x \mid x\in S\}$ and
$\rsec{S}\coloneqq \{ R_x\mid x\in S\}$ denote the \emph{left} and \emph{right sections} of $S$.

For a subloop $S$ of a loop $Q$, the \emph{left} and \emph{right relative multiplication groups} and the \emph{relative multiplication group} are permutation groups generated by the sections:
\[
\lmlt{Q;S}\coloneqq \langle \lsec{S}\rangle\,,\quad \rmlt{Q;S}\coloneqq \langle \rsec{S}\rangle\,,
\quad \mlt{Q;S}\coloneqq \langle \lsec{S},\rsec{S} \rangle\,.
\]
In case $S=Q$, these are just called the \emph{left} and \emph{right multiplication groups} and the \emph{multiplication group} of $Q$,
respectively, and denoted more simply by $\lmlt{Q}$, $\rmlt{Q}$ and $\mlt{Q}$. The \emph{left} and \emph{right inner mapping group} $\linn{Q}$ and $\rinn{Q}$ and the \emph{inner mapping group} $\inn{Q}$ are the stabilizers of $e$ in the corresponding multiplication groups.

A subloop of a loop $Q$ is \emph{normal} if it is a block of $\mlt{Q}$. In particular, if $H$ is a normal subgroup of $\mlt{Q}$, then
the orbit of $e$ under $H$ is a normal subloop. Two other useful characterizations of normality are the following: a subloop $S$ of loop
$Q$ is normal if and only if $S$ is invariant under the action of $\inn{Q}$ if and only if $S$ is a block of a congruence of $Q$.

In a loop $Q$, the \emph{left nucleus}, \emph{middle nucleus} and \emph{right nucleus} are defined, respectively, by
\begin{align*}
  \lnuc{Q} &\coloneqq \{ a\in Q\mid ax\cdot y = a\cdot xy\,,\ \forall x,y\in Q\}\,, \\
  \mnuc{Q} &\coloneqq \{ a\in Q\mid xa\cdot y = x\cdot ay\,,\ \forall x,y\in Q\}\,, \\
  \rnuc{Q} &\coloneqq \{ a\in Q\mid xy\cdot a = x\cdot ya\,,\ \forall x,y\in Q\}\,.
\end{align*}
These have various useful characterizations which are immediate from the definitions:
\begin{align*}
\lnuc{Q} &= \{a\in Q\mid L_a L_x = L_{ax}\,,\ \forall x\in Q\} = \{a\in Q\mid L_a R_y = R_y L_a\,,\ \forall y\in Q\}\,,\\
\mnuc{Q} &= \{a\in Q\mid L_x L_a = L_{xa}\,,\ \forall x\in Q\} = \{a\in Q\mid R_y R_a = R_{ay}\,,\ \forall y\in Q\}\,,\\
\rnuc{Q} &= \{a\in Q\mid R_a R_y = R_{ya}\,,\ \forall y\in Q\} = \{a\in Q\mid L_x R_a = R_a L_x\,,\ \forall x\in Q\}\,.
\end{align*}
We will need their pairwise intersections, so we introduce the notation
\begin{align*}
\lmnuc{Q} &\coloneqq \lnuc{Q}\,\cap\, \mnuc{Q}\,,\\
\lrnuc{Q} &\coloneqq \lnuc{Q}\,\cap\, \rnuc{Q}\,,\\
\rmnuc{Q} &\coloneqq \rnuc{Q}\,\cap\, \mnuc{Q}\,.
\end{align*}
Finally, the nucleus, $\nuc{Q}$ is defined to be the intersection of all three nuclei. All of the sets defined above are subloops
of any loop $Q$; however, none of them need be normal subloops.

The \emph{commutant} (also known as the \emph{centrum}, \emph{semicenter}, \emph{commutative center} and other names) of a loop $Q$ is the subset
\[
C(Q)\coloneqq \{a\in Q\mid ax=xa\,,\ \forall x\in Q\} = \{a\in Q\mid L_a = R_a\}\,.
\]
In general, $C(Q)$ is not a subloop of $Q$, even in structured varieties like Bol loops \cite{KPV} or C loops (\cite{SAS}, Ex.~4.1).

Finally, the \emph{center} of $Q$ is $Z(Q) = C(Q)\cap \nuc{Q}$. This is a normal subloop and, in fact, is precisely the fixed point set of $\inn{Q}$. The center can be characterized as the intersection of the commutant with any pair of nuclei:
\[
Z(Q) = C(Q)\cap \mathrm{Nuc}_{i,j}{Q},
\]
where $i,j\in \{\ell,m,r\}$, $i\neq j$.

Note that in case $S$ is a subloop of the left or middle nucleus of a loop $Q$, then $\lmlt{Q;S} = \lsec{S}$, while if $S$ is a subloop
of the middle or right nucleus, then $\rmlt{Q;S} = \rsec{S}$. These will be the only relative one-sided multiplication groups encountered in this paper.

A triple $(\alpha,\beta,\gamma)$ of bijections $\alpha,\beta,\gamma\colon Q\to Q$ of a loop $Q$ is said to be an \emph{autotopism}
if, for all $x,y\in Q$, $\alpha x\cdot \beta y = \gamma (x\cdot y)$. In the special case $\alpha=\beta=\gamma$, the autotopism
is identified with the underlying \emph{automorphism} $\alpha$. The set $\atp{Q}$ of all autotopisms of $Q$ forms a group under
composition of triples of mappings.

Autotopisms of $Q$ in which one of the three permutations is the identity mapping $\id{Q}$ can be completely described in
terms of the nuclei \cite{Br}:

\begin{prp}\label{Prp:id_nuc}
Let $Q$ be a loop and let $\alpha,\beta,\gamma\colon Q\to Q$ be bijections.
\begin{enumerate}
  \item If $(\alpha,\id{Q},\gamma)\in \atp{Q}$, then $\alpha = \gamma = L_a$ where $a = \alpha e\in \lnuc{Q}$.
  \item If $(\alpha,\beta,\id{Q})\in \atp{Q}$, then $\alpha = R_a\inv$ and $\beta = L_{a\rinv}\inv$ where $a = \beta e\in \mnuc{Q}$.
  \item If $(\id{Q},\beta,\gamma)\in \atp{Q}$, then $\beta = \gamma = R_a$ where $a = \gamma e\in \rnuc{Q}$.
\end{enumerate}
In particular,
\begin{align*}
\lnuc{Q} &= \{a\in Q\mid (L_a,\id{Q},L_a)\in \atp{Q}\}\,,\\
\mnuc{Q} &= \{a\in Q\mid (R_a\inv,L_{a\rinv}\inv,\id{Q})\in \atp{Q}\}\,,\\
\rnuc{Q} &= \{a\in Q\mid (\id{Q},R_a,R_a)\in \atp{Q}\}\,.
\end{align*}
\end{prp}

%%%%%%

A loop $Q$ is \emph{left alternative} if, for all $x,y\in Q$,
\[
x\cdot xy = x^2\cdot y\qquad\text{or equivalently,}\qquad L_x^2 = L_{x^2}\,. \tag{\textsc{LAlt}}\label{Eqn:LAlt}
\]
Dually, a loop $Q$ is \emph{right alternative} if, for all $x,y\in Q$,
\[
xy\cdot y = xy^2\qquad\text{or equivalently,}\qquad R_y^2 = R_{y^2}\,. \tag{\textsc{RAlt}}\label{Eqn:RAlt}
\]
A loop which both left and right alternative is simply called \emph{alternative}.

A loop $Q$ is said to have the \emph{left inverse property} (LIP) if, for all $x,y\in Q$,
\[
x\linv \cdot xy = y\qquad\text{or equivalently,}\qquad L_x\inv = L_{x\linv}\,. \tag{\textsc{LIP}}\label{Eqn:LIP}
\]
Dually, a loop $Q$ is said to have the \emph{right inverse property} (RIP) if, for all $x,y\in Q$,
\[
xy\cdot y\rinv = x\qquad\text{or equivalently,}\qquad R_y\inv = R_{y\rinv}\,. \tag{\textsc{RIP}}\label{Eqn:RIP}
\]
A loop $Q$ is said to have the \emph{antiautomorphic inverse property} (AAIP) if, for all $x,y\in Q$,
\[
(xy)\linv = y\linv x\linv\qquad\text{or equivalently,}\qquad (xy)\rinv = x\rinv y\rinv\,. \tag{\textsc{AAIP}}\label{Eqn:AAIP}
\]
A loop $Q$ has the \emph{automorphic inverse property} (AIP) if, for all $x,y\in Q$,
\[
(xy)\rinv = x\rinv y\rinv\qquad\text{or equivalently,}\qquad (xy)\linv = x\linv y\linv\,. \tag{\textsc{AIP}}\label{Eqn:AIP}
\]

If $Q$ has the LIP, then for all $x\in Q$, $x\linv = x\linv\cdot xx\rinv = x\rinv$. A dual argument applies if $Q$ has the RIP.
If $Q$ has AAIP, then for all $x\in Q$, $xx\linv = (x\rinv)\linv x\linv = (xx\rinv)\linv = e$, and so $x\linv = x\rinv$.
Thus a loop satisfying any of LIP, RIP or AAIP has two-sided inverses.

A loop $Q$ satisfying any two of LIP, RIP and AAIP is easily seen to satisfy the third property, and in that case, $Q$ is said to have the \emph{inverse property} (IP).

\begin{exm}
Unlike the other aforementioned properties involving inverses, the AIP does not imply that a loop has two-sided inverses. Here is a Cayley table of an AIP loop in which $1^l = 2 \ne 3 = 1^r$.
\[
\begin{array}{ccccc}
0&1&2&3&4\\
1&4&3&0&2\\
2&0&4&1&3\\
3&2&0&4&1\\
4&3&1&2&0
\end{array}
\]
\end{exm}

The first three parts of the following are well known; the fourth part, although known, is less familiar.

\begin{lem}\label{Lem:lmrnuc}
Let $Q$ be a loop.
\begin{enumerate}
  \item If $Q$ has the LIP, then $\lnuc{Q} = \mnuc{Q}$.
  \item If $Q$ has the RIP, then $\rnuc{Q} = \mnuc{Q}$.
  \item If $Q$ has the AAIP, then $\lnuc{Q} = \rnuc{Q}$.
  \item If $Q$ has the AIP, then $\mnuc{Q}\subseteq C(Q)$ and $\lmnuc{Q} = \lrnuc{Q} = \rmnuc{Q} = Z(Q)$.
\end{enumerate}
\end{lem}
\begin{proof}
(1) and (2) are easy to check directly, but we also note that they are corollaries of the more general Theorem \ref{Thm:iso_nucs}
below (and the theorem's dual).

(3) The opposite loop $(Q,\bullet)$ defined by $x \bullet y = yx$ satisfies $\lnuc{Q,\bullet} = \rnuc{Q,\cdot}$.
The AAIP just says that $\lambda\colon Q\to Q;x\mapsto x\linv$ is an isomorphism of $(Q,\cdot)$ onto $(Q,\bullet)$,
hence preserves all nuclei. But the nuclei of $(Q,\cdot)$ are invariant under $\lambda$, so
$\lnuc{Q,\cdot} = \lambda\lnuc{Q,\cdot} = \lnuc{Q,\bullet} = \rnuc{Q,\cdot}$.

(4) For $a\in \mnuc{Q}$, $x\in Q$, $x\linv a\inv\cdot ax = x\linv\cdot a\inv a\cdot x = x\linv x = e$. Thus
$ax = (x\linv a\inv)\rinv = xa$ using the AIP. Thus $a\in C(Q)$.

For the remaining assertion, we need only check that $\lrnuc{Q}\subseteq C(Q)$. For $a\in \lrnuc{Q}$, $x\in Q$, we have
$(xa\cdot x\rinv)a\inv = xa\cdot x\rinv a\inv = xa\cdot (xa)\rinv = e$, using $a\inv\in \rnuc{Q}$ and the AIP.
Thus $xa\cdot x\rinv = a = a\cdot xx\rinv = ax\cdot x\rinv$ since $a\in \lnuc{Q}$. Cancelling $x\rinv$ on the right,
we get $xa = ax$. Thus $a\in C(Q)$.
\end{proof}

We say that a loop has \textit{left (middle, right) nuclear squares} if all squares are in the left (middle, right) nucleus. These are varieties of loops defined by the identities \eqref{Eqn:LNS}, \eqref{Eqn:MNS}, and \eqref{Eqn:RNS}, respectively.

%Commenting out what's below since I don't think we need it.
%
%\begin{lem}
%Let $Q$ be a left inverse property loop and let $a\in \lnuc{Q}$. For all $x\in L$,
%\begin{align}
%xa\cdot a\inv &= x      \label{Eqn:xaa'=x} \\
%(ax)\inv &= x\inv a\inv \label{Eqn:lnuc_aaip} \\
%(xa)\inv &= a\inv x\inv \label{Eqn:mnuc_aaip}
%\end{align}
%\end{lem}
%
%\begin{proof}
%Since $a\in \mnuc{Q}$ by Lemma \ref{Lem:lmrnuc}, \eqref{Eqn:xaa'=x} is immediate.
%For \eqref{Eqn:lnuc_aaip}, we have
%\[
%ax\cdot x\inv a\inv = a(x\cdot x\inv a\inv) = aa\inv = 1\,,
%\]
%using $a\in \lnuc{Q}$ and the left inverse property. For \eqref{Eqn:mnuc_aaip},
%\[
%xa\cdot a\inv x\inv = x(a\cdot a\inv x\inv) = xx\inv = 1\,,
%\]
%using $a\in \mnuc{Q}$ and the left inverse property.
%\end{proof}

\section{Principal isostrophes}
\label{Sec:isostrophes}

Associated to any loop $(Q,\cdot)$ are two other useful loops defined by the binary operations $x\ast y\coloneq x\linv \ldv y$ and $x\circ y\coloneq x\rdv y\rinv$. These loops are particular cases of isostrophes of $Q$ \cite{Pf};
we will refer to $(Q,\ast)$ and $(Q,\circ)$ as the \emph{principal left} and \emph{right isostrophes} of $Q$, respectively. The principal isostrophes have the same identity element as $(Q,\cdot)$. These loops reflect various structural features of $Q$ itself. For example,
$Q$ has the LIP (resp. RIP) if and only if $x\ast y = x\cdot y$ (resp. $x\circ y = x\cdot y$) for all $x,y\in Q$.

There being a clear duality between the principal left and right isostrophes of a loop $(Q,\cdot)$, we focus only on the
right isostrophe $(Q,\circ)$, leaving dual statements to the reader. The division operations of the principal right isostrophe
are given by
\[
x \ldv\!\ldv y = (y\ldv x)\linv \qquad\text{and}\qquad x \rdv\!\rdv y = xy\rinv
\]
for all $x,y\in Q$. It follows that left and right inverses in $(Q,\circ)$ are, respectively,
\[
x^{(\ell)}\coloneq e \rdv\!\rdv x = x\rinv\qquad\text{and}\qquad x^{(r)}\coloneq x\ldv\!\ldv e = x\linv
\]
for all $x\in Q$. We will denote the left and right translation maps in $(Q,\circ)$ by $L_x^{\circ}$ and
$R_x^{\circ}$ for $x\in Q$.

\begin{lem}\label{Lem:rder2}
Let $(Q,\cdot)$ be a loop with principal right isostrophe $(Q,\circ)$. Then the principal right isostrophe
of $(Q,\circ)$ is $(Q,\cdot)$.
\end{lem}
\begin{proof}
Indeed, for all $x,y\in Q$, $x\rdv\!\rdv y^{(r)} = x(y^{(r)})\rinv = x(y\linv)\rinv = xy$.
\end{proof}

Because the defining operations of each of the loops $(Q,\cdot)$ and $(Q,\circ)$ can be expressed in terms of
the other's operations, it follows that both loops have the same congruences, hence the same normal subloops.
We record this observation for later reference.

\begin{lem}\label{Lem:same}
Let $(Q,\cdot)$ be a loop with principal right isostrophe $(Q,\circ)$, and let $S\subseteq Q$. Then $S$
is a normal subloop of $(Q,\cdot)$ if and only if it is a normal subloop of $(Q,\circ)$.
\end{lem}

Next we examine the relationship between the autotopism groups of $(Q,\cdot)$ and $(Q,\circ)$.

\begin{lem}\label{Lem:iso_atp}
Let $(Q,\cdot)$ be a loop with principal right isostrophe $(Q,\circ)$. For bijections $\alpha,\beta,\gamma\colon Q\to Q$,
$(\alpha,\beta,\gamma)\in \atp{Q,\circ}$ if and only if $(\gamma,\rho\beta\lambda,\alpha)\in \atp{Q,\cdot}$.
\end{lem}
\begin{proof}
We have $(\alpha,\beta,\gamma)\in \atp{Q,\circ}$ if and only if $\alpha x \rdv (\beta y)\rinv = \gamma(x\rdv y\rinv)$ for
all $x,y\in Q$. Replace $x$ with $xy\rinv$, multiply on the right by $(\beta y)\rinv$, and then replace $y$ with $y\linv$.
It follows that $(\alpha,\beta,\gamma)\in \atp{Q,\circ}$ if and only if $\alpha(xy) = \gamma x\cdot (\beta y\linv)\rinv$
for all $x,y\in Q$, that is, if and only if $(\gamma,\rho\beta\lambda,\alpha)\in \atp{Q,\cdot}$.
\end{proof}

\begin{thm}\label{Thm:iso_nucs}
Let $(Q,\cdot,e)$ be a loop. Then:
\begin{enumerate}
\item $\lnuc{Q,\cdot} = \lnuc{Q,\circ}$;
\item $\mnuc{Q,\cdot} = \rnuc{Q,\circ}$.
\end{enumerate}
\end{thm}
\begin{proof}
(1) By Proposition \ref{Prp:id_nuc}, $a\in \lnuc{Q,\cdot}$ if and only if $(L_a,\id{Q},L_a)\in \atp{Q,\cdot}$.
By Lemma \ref{Lem:iso_atp}, this holds if and only if $(L_a,\lambda\id{Q}\rho,L_a)= (L_a,\id{Q},L_a)\in \atp{Q,\circ}$.
By Proposition \ref{Prp:id_nuc}, this holds if and only if $(L_a^{\circ},\id{Q},L_a^{\circ})\in \atp{Q,\circ}$, that is,
if and only if $a\in \lnuc{Q,\circ}$.

(2) By Proposition \ref{Prp:id_nuc}, $a\in \mnuc{Q,\cdot}$ if and only if $(R_{a\rinv}\inv,L_a\inv,\id{Q},)\in \atp{Q,\cdot}$.
By Lemma \ref{Lem:iso_atp}, this holds if and only if $(\id{Q},\lambda L_a\inv \rho,R_{a\rinv})\in \atp{Q,\circ}$.
By Proposition \ref{Prp:id_nuc}, this holds if and only if $(\id{Q},R_a^{\circ},R_a^{\circ})\in \atp{Q,\circ}$, that is,
if and only if $a\in \rnuc{Q,\circ}$.
\end{proof}

\section{Loops with squares in two nuclei}
\label{Sec:two_nuc}

In this section we turn to the main loop varieties of interest in this paper: loops with squares in two nuclei. The case of middle and right nuclear squares is obviously dual to the case of left and middle nuclear squares, so we will only consider loops with left and middle nuclear squares and loops with left and right nuclear squares. The main result of the section is that the corresponding intersection of nuclei is a normal subloop.

\begin{lem}\label{Lem:sq_ids}
Let $Q$ be a loop and let $a\in Q$. If $a^2\in \lnuc{Q}$, then
\begin{equation}\label{Eqn:x2xlx}
    a^2 a\linv = a\qquad\text{and}\qquad L_{a^2} L_{a\linv} = L_a\,.
\end{equation}
\end{lem}
\begin{proof}
We have $a^2 a\linv\cdot a = a^2\cdot a\linv a = a^2$ . Canceling $a$ on the right, we obtain $a^2 a\linv = a$. Now since $a^2\in \lnuc{Q}$, $L_{a^2} L_{a\linv} = L_{a^2 a\linv} = L_a$.
\end{proof}

\begin{thm}\label{Thm:lm=lr}
Let $(Q,\cdot)$ be a loop with principal right isostrophe $(Q,\circ)$. Then:
\begin{enumerate}
  \item $(Q,\cdot)$ has left nuclear squares if and only if $(Q,\circ)$ has left nuclear squares;
  \item $(Q,\cdot)$ has left and middle nuclear squares if and only if $(Q,\circ)$ has left and right nuclear squares.
\end{enumerate}
\end{thm}
\begin{proof}
(1) Assume $(Q,\cdot)$ has left nuclear squares. For all $x\in Q$, $x^2\cdot (x^2)\inv x = x = x^2 x\linv$, using \eqref{Eqn:x2xlx}.
Cancelling, we have $(x^2)\inv x = x\linv$, or equivalently, $(x^2)\inv = x\linv \rdv x = x\linv\circ x\linv$.
Thus $x\circ x = ((x\rinv)^2)\inv$ for all $x\in Q$. By Theorem \ref{Thm:iso_nucs}(1), $(Q,\circ)$ has left nuclear squares.
The converse follows from Lemma \ref{Lem:rder2}.

(2) This follows from (1) and Theorem \ref{Thm:iso_nucs}.
\end{proof}

In view of Theorem \ref{Thm:iso_nucs}(2), it is natural to wonder if Theorem \ref{Thm:lm=lr}(2) can be improved by dropping the conditions on left nuclear squares.

\begin{exm}
The following tables show a loop $(Q,\cdot)$ and its principal right isostrophe $(Q,\circ)$. Here $\mnuc{Q,\cdot} = \rnuc{Q,\circ} = \{1,2\}$. From examining the main diagonals of each table, we see that $(Q,\cdot)$ has middle nuclear squares but $(Q,\circ)$ does not have right nuclear squares.
\[
\begin{array}{c|cccccc}
\cdot & 1 & 2 & 3 & 4 & 5 & 6 \\
\hline
1 & 1 & 2 & 3 & 4 & 5 & 6\\
2 & 2 & 1 & 4 & 3 & 6 & 5\\
3 & 3 & 5 & 1 & 6 & 4 & 2\\
4 & 4 & 6 & 5 & 2 & 1 & 3\\
5 & 5 & 3 & 6 & 1 & 2 & 4\\
6 & 6 & 4 & 2 & 5 & 3 & 1
\end{array}
\qquad
\begin{array}{c|cccccc}
\circ & 1 & 2 & 3 & 4 & 5 & 6 \\
\hline
1 & 1 & 2 & 3 & 4 & 5 & 6\\
2 & 2 & 1 & 6 & 5 & 4 & 3\\
3 & 3 & 5 & 1 & 6 & 2 & 4\\
4 & 4 & 6 & 2 & 3 & 1 & 5\\
5 & 5 & 3 & 4 & 1 & 6 & 2\\
6 & 6 & 4 & 5 & 2 & 3 & 1
\end{array}
\]
\end{exm}

The first main result of this section is the following.

\begin{thm}
\label{Thm:LMNuc_normal}
Let $Q$ be a loop with left and middle nuclear squares and let $N\coloneq \lmnuc{Q}$ Then $L_{(N)}$ is a normal subgroup of $\mlt{Q}$ and $N$ is a normal subloop of $Q$.
\end{thm}
\noindent The proof requires a few lemmas.

\begin{lem}\label{Lem:sq_ids2}
Let $Q$ be a loop with left and middle nuclear squares. For all $x\in Q$,
\begin{align}
    L_{x\linv} L_{x^2}  &= L_{x^{\ell\ell}}                                 \label{Eqn:xlx2xll} \\
    L_{x\linv} L_x      &= L_x L_{x\rinv}                                   \label{Eqn:xlxxxr} \\
    L_x L_{x\rinv}\inv  &= L_{xx^{rr}}\text{ and } xx^{rr}\in \lmnuc{Q}     \label{Eqn:xxrr}
\end{align}
\end{lem}
\begin{proof}
We have $x\linv x^2\cdot x\linv = x\linv \cdot x^2 x\linv = x\linv x = e$ using $x^2\in \mnuc{Q}$ and \eqref{Eqn:x2xlx}.
Thus $x\linv x^2 = x^{\ell\ell}$.
Now since $x^2\in \mnuc{Q}$, $L_{x\linv} L_{x^2} = L_{x\linv x^2} = L_{x^{\ell\ell}}$. This proves \eqref{Eqn:xlx2xll}.

Now we compute
\[
L_{x\linv} L_x \byeqn{Eqn:x2xlx} L_{x\linv} L_{x^2} L{x\linv} \byeqn{Eqn:xlx2xll} L_{x^{\ell\ell}} L_{x\linv}\,.
\]
Replacing $x$ with $x\rinv$ completes the proof of \eqref{Eqn:xlxxxr}.

For the second claim in \eqref{Eqn:xxrr}, we use \eqref{Eqn:x2xlx} to compute $(x\rinv)^2\cdot xx^{rr} = x^r x^rr = e$, and
thus $xx^{rr} = ((x\rinv)^2)\inv\in \lmnuc{Q}$. Using this, we compute $(xx^{rr}\cdot x^r)x^{rr} = xx^{rr}\cdot x^r x^{rr} = xx^{rr}$;
cancelling $x^{rr}$ on the right gives $xx^{rr}\cdot x^r = x$. Again using $xx^{rr}\in \lmnuc{Q}$, we have $L_{xx^{rr}} L_{x\rinv} =
L_{xx^{rr}\cdot x\rinv} = L_x$. Rearranging, we have proved \eqref{Eqn:xxrr}.
\end{proof}

\begin{lem}\label{Lem:xax}
Let $Q$ be a loop with left and middle nuclear squares. For all $a\in \lmnuc{Q}$ and for all $x\in Q$, $xax\in \lmnuc{Q}$.
\end{lem}
\begin{proof}
We compute
\[
L_a L_{xax} L_y = L_{(ax)^2} L_y = L_{(ax)^2 y} = L_{a\cdot xax\cdot y} = L_a L_{xax\cdot y}\,,
\]
using $a\in \lmnuc{Q}$ in the first, third and fourth equalities, and left nuclear squares in the second.
Thus $xax\in \lnuc{Q}$.

Now
\begin{align*}
L_y L_{xax} &= L_{y\rdv a\cdot a} L_{xax} = L_{y\rdv a} L_a L_{xax} = L_{y\rdv a} L_{(ax)^2} \\
&= L_{(y\rdv a)(ax)^2} = L_{(y\rdv a)a\cdot xax} = L_{y\cdot xax}\,,
\end{align*}
using $a\in \lmnuc{Q}$ in the second, third and fifth equalites, and middle nuclear squares in the fourth.
Thus $xax\in \mnuc{Q}$.
\end{proof}

\begin{lem}\label{Lem:Lconj}
Let $Q$ be a loop with left and middle nuclear squares. For all $a\in \lmnuc{Q}$ and for all $x\in Q$,
\begin{enumerate}
  \item $L_x L_a L_x\inv = L_{xax\rinv}$ and $xax\rinv\in \lmnuc{Q}$;
  \item $L_x\inv L_a L_x = L_{x\ldv (ax)}$ and $x\ldv (ax)\in \lmnuc{Q}$.
\end{enumerate}
\end{lem}
\begin{proof}
(1) First we prove
\begin{equation}\label{Eqn:Lconj-1}
L_{x\rinv a x\rinv} L_{xa\inv} = L_{x\rinv}\,.
\end{equation}
Indeed,
\[
L_a L_{x\rinv a x\rinv} L_{xa\inv} = L_{(ax\rinv)^2} L_{(ax\rinv)\linv} = L_{ax\rinv} = L_a L_{x\rinv}\,,
\]
using $a\in \lmnuc{Q}$ and \eqref{Eqn:x2xlx}. Canceling on $L_a$ gives \eqref{Eqn:Lconj-1}.

Now we compute
\begin{align*}
L_x L_a L_x\inv &= L_x (L_x L_{a\inv})\inv = L_x L_{xa\inv}\inv
= L_{(x\rinv)\linv} L_{x\rinv}\inv\cdot L_{x\rinv} L_{xa\inv}\inv \\
&= L_{((x\rinv)^2)\inv} L_{x\rinv a x\rinv}  = L_u
\end{align*}
where $u = ((x\rinv)^2)\inv\cdot x\rinv a x\rinv$,
using $a\inv\in \mnuc{Q}$ in the second equality, \eqref{Eqn:x2xlx} and \eqref{Eqn:Lconj-1} in the fourth, and left nuclear squares (or Lemma \ref{Lem:xax}) in the fifth. Our assumption on squares,
together with Lemma \ref{Lem:xax}, imply that $u\in \lmnuc{Q}$.
If we apply both sides of the preceding calculation to $1$, we get
$xax\rinv = u$ as desired.

(2) Recalling that $x\ldv a = (a\inv x)\rinv$, we have
\[
L_{x\ldv a} = L_{(a\inv x)\rinv} \byeqn{Eqn:x2xlx} L_{((a\inv x)\rinv)^2} L_{a\inv x}
= L_{((a\inv x)\rinv)^2} L_{a\inv} L_x = L_v L_x = L_{vx}\,,
\]
using $a\inv\in \mnuc{Q}$ in the third equality and $v = ((a\inv x)\rinv)^2\cdot a\inv\in \lmnuc{Q}$
in the fourth. Thus $a = xvx$ and so $xv = a\rdv x = ax\linv$.

Next,
\[
L_x L_v L_x = L_{xv} L_x = L_{ax\linv} L_x
= L_a L_{x\linv} L_x \byeqn{Eqn:xlxxxr} L_a L_x L_{x\rinv}\,,
\]
using the remark two lines above in the second equality and $a\in \lnuc{Q}$ in the third. Thus
\[
L_x\inv L_a L_x = L_v L_x L_{x\rinv}\inv \byeqn{Eqn:xxrr} L_v L_{xx^{rr}} \byeqn{Eqn:xxrr} L_{v\cdot xx^{rr}}\,.
\]
Applying both sides to $e$, we get $x\ldv ax = v\cdot xx^{rr}\in \lmnuc{Q}$, and thus
$L_x\inv L_a L_x = L_{x\ldv ax}$. This completes the proof of (2).
\end{proof}

We are ready for the following.

\begin{proof}[Proof of Theorem \ref{Thm:LMNuc_normal}]
(1) For any $a\in \lnuc{Q}$, $L_a R_x = R_x L_a$ for all $x\in Q$. Thus $\rmlt{Q}$ centralizes $L_{(N)}$. By Lemma \ref{Lem:Lconj},
$\lmlt{Q}$ normalizes $L_{(N)}$. This establishes the normality of $L_{(N)}$ in $\mlt{Q}$.

(2) This follows immediately from (1) since $N = \lmnuc{Q}$ is the orbit of $1$ under the action of the normal subgroup $L_{(N)}$ of $\mlt{Q}$.
\end{proof}

Our second main result of the section follows.

\begin{thm}
\label{Thm:LRNuc_normal}
Let $Q$ be a loop with left and right nuclear squares. Then $\lrnuc{Q}$ is a normal subloop of $Q$.
\end{thm}
\begin{proof}
Since $(Q,\cdot)$ has left and right nuclear squares, Lemma \ref{Lem:rder2} and Theorem \ref{Thm:lm=lr}, the principal right isostrophe $(Q,\circ)$ has left and middle nuclear squares. By Theorem \ref{Thm:LMNuc_normal},
$\lmnuc{Q,\circ}$ is a normal subloop of $(Q,\circ)$. Since $\lrnuc{Q,\cdot} = \lmnuc{Q,\circ}$ by Theorem \ref{Thm:iso_nucs}(2), it follows from Lemma \ref{Lem:same} that $\lrnuc{Q,\cdot}$ is normal in $(Q,\cdot)$.
\end{proof}

\begin{cor}\label{Cor:LMNS_simple}
Let $Q$ be a simple loop with all squares in two nuclei. Then either $Q$ is a group or $Q$ is a nonassociative loop of exponent two.
\end{cor}
\begin{proof}
Let $N$ denote the intersection of the corresponding nuclei. By Theorem \ref{Thm:LMNuc_normal}, its dual, or Theorem \ref{Thm:LRNuc_normal}, whichever is appropriate, $N$ is a normal subloop of $Q$. By simplicity, either $N = Q$ or $N = \{e\}$. These are precisely the two cases in the statement.
\end{proof}

\begin{cor}\label{Cor:nuc_sq}
Let $Q$ be a loop with nuclear squares. Then $\nuc{Q}$ is a normal subloop of $Q$.
\end{cor}
\begin{proof}
This follows immediately from Theorems \ref{Thm:LMNuc_normal} and \ref{Thm:LRNuc_normal}.
\end{proof}

\section{Loops with central squares}
\label{Sec:AIP}

In this section we specialize from loops with nuclear squares to loops with central squares, that is, loops with both nuclear and commuting squares. We will then specialize further to consider loops in which the squaring map $x\mapsto x^2$ is a centralizing endomorphism.

A loop $Q$ is said to be \emph{power-associative} if, for each $x\in Q$, the subloop $\langle x\rangle$ is a group. Informally, power-associativity means that powers of elements are defined unambiguously. For now, we fix a convention for powers, say, $x^n := L_x^n(1)$ for every integer $n$. Power-associativity is then equivalent to $x^m\cdot x^n = x^{m+n}$ for all $m,n\in \mathbb{Z}$ and all $x\in Q$.

\begin{lem}\label{Lem:CS_PA}
Every loop with central squares is power-associative.
\end{lem}
\begin{proof}
Let $x\in Q$. Since $x^2\in Z(Q)$, note that $(x^2)^k\in Z(Q)$ for every $k\in \mathbb{Z}(Q)$. So we first show 
\begin{equation}\label{Eq:PA-tmp1}
	x^{2k} = (x^2)^k\in Z(Q)
\end{equation}
for each $k\in \mathbb{Z}$. This is clear for $k=0$. If \eqref{Eq:PA-tmp1} holds
for some $k\geq 0$, then
\[
x^{2(k+1)} = L_x^{2k+2}(1) = L_x^{2k}L_x^2(1) = L_x^{2k}(x^2) = L_x^{2k}(1)\cdot x^2 = x^{2k}x^2 = (x^2)^{k+1},
\]
using centrality of $x^2$ in the fourth equality and the inductive hypothesis in the fifth. Next,
\[
x^{-2k}x^{2k} = x^{2k} x^{-2k} = x^{2k}\cdot L_x^{-2k}(1) = L_x^{-2k}(x^{2k}) = 1\,,
\]
using $x^{2k}\in Z(Q)$ in the second equality. Thus $x^{-2k} = (x^{2k})^{-1}
= ((x^2)^k)^{-1} = (x^2)^{-k}$. This establishes \eqref{Eq:PA-tmp1} for all $k\in \mathbb{Z}$.

From \eqref{Eq:PA-tmp1}, we immediately get 
\begin{equation}\label{Eq:PA-tmp2}
	x^{2k}x^{2\ell} = x^{2(k+\ell)}
\end{equation}
for all $k,\ell\in \mathbb{Z}$. 

Now we prove $x^m x^n = x^{m+n}$ for all $m,n\in \mathbb{Z}$. 
We have $m=2k+i$, $n=2\ell+j$ for some $k,\ell\in \mathbb{Z}$, $i,j\in \{0,1\}$.
Then by \eqref{Eq:PA-tmp2},
\[
x^m x^n = L_x^{2k+i}(1)\cdot L_x^{2\ell+j}(1) 
= L_x^i(x^{2k})\cdot L_x^j(x^{2\ell}) = L_x^i(1)\cdot L_x^j(1)\cdot x^{2k}x^{2\ell} = x^i x^j \cdot x^{2(k+\ell)}\,.
\]
If $i=j=1$, then $x^m x^n = x^2 x^{2(k+\ell)} = x^{2(k+\ell+1)} = x^{m+n}$
by \eqref{Eq:PA-tmp2}. Otherwise, $x^m x^n = L_x^{i+j}x^{2(k+\ell)} = x^{2(k+\ell)+i+j} = x^{m+n}$. This completes the proof.
\end{proof}

\begin{lem}\label{Lem:CS_ident}
Let $Q$ be a loop with central squares. For all $x,y\in Q$,
\begin{equation}\label{AIP_SqEndo_tmp1}
	x^2 y^2 = xy\cdot (x\inv y\inv)\inv\,.
\end{equation}
\end{lem}
\begin{proof}
By Lemma \ref{Lem:CS_PA}, $Q$ is power-associative. Using this and the centrality of squares, we have
\[
x^2 y^2 = x^2 y^2\cdot x\inv y\cdot (x\inv y)\inv = x^2 x\inv y\cdot (x\inv y y^{-2})\inv = xy\cdot (x\inv y\inv)\inv\,.
\]
\end{proof}

\begin{lem}\label{Lem:CS_AIP_SqEnd}
Let $Q$ be a loop with central squares. Then $Q$ has the automorphic inverse property if and only if the squaring map $s\colon Q\to Q;x\mapsto x^2$ is an endomorphism.
\end{lem}
\begin{proof}
If the AIP holds, then the right hand side of \eqref{AIP_SqEndo_tmp1} equals $(xy)^2$ and thus $s$ is an endomorphism. Conversely, if $s$ is an endomorphism, then the left hand side \eqref{AIP_SqEndo_tmp1} equals $(xy)^2$; cancelling $xy$ on the left gives $xy = (x\inv y\inv)\inv$, which is the AIP.
\end{proof}

\begin{thm}\label{Thm:AIP_SqEndo}
Let $Q$ be a loop with squares in two nuclei. The following are equivalent:
\begin{enumerate}
\item $Q$ has the automorphic inverse property;
\item The squaring map $s:Q \to Q; x\mapsto  x^2$ is an endomorphism.
\end{enumerate}
When these conditions hold, $Q$ has central squares.
\end{thm}
\begin{proof}
By Lemma \ref{Lem:CS_AIP_SqEnd}, it is sufficient to prove that each of (1) and (2) imply that squares are central.

Assume (1). By Lemma \ref{Lem:lmrnuc}(4), the pairwise intersections of the nuclei all coincide with the center. Since squares are contained in two nuclei, it follows that squares are central.

Assume (2). There are three cases to consider, depending upon which pairs of nuclei contain all squares.

First assume $Q$ has left and middle nuclear squares. We will prove that $\lmnuc{Q}\subseteq C(Q)$, which implies $\lmnuc{Q} = Z(Q)$, and so $Q$ will have central squares. Let $a\in \lmnuc{Q}$. For all $x\in Q$,
\[
a(ax\cdot x) = a^2 x\cdot x = a^2 x^2 = ax\cdot ax = a(x\cdot ax) = a(xa\cdot x)\,,
\]
using $a\in \lnuc{Q}$ in the first and fourth equality, $a^2\in \lnuc{Q}$ in the second, endomorphic squaring in the third, and $a\in \mnuc{Q}$ in the fifth. Cancelling $a$ on the left and then $x$ on the right, we obtain $ax\cdot xa$ for all $x\in Q$, that is, $a\in C(Q)$. This completes the proof of this case. 

The case where $Q$ has middle and right nuclear squares is dual to the preceding case, hence omitted.

Finally, assume $Q$ has left and right nuclear squares. We will prove that $\lrnuc{Q}\subseteq C(Q)$, which implies $\lrnuc{Q} = Z(Q)$, and so $Q$ will have central squares. Let $a\in \lrnuc{Q}$. For all $x\in Q$,
\[
a\cdot x^2 a \cdot x^2 = ax^2\cdot ax^2 = a^2\cdot x^2 x^2 = a\cdot ax^2\cdot x^2\,,
\]
using both $a\in \lnuc{Q}$ and $x^2\in \rnuc{Q}$ in the first and third equalities. Cancelling, we have
\begin{equation}\label{Eqn:AIP_SqEndo-tmp1}
ax^2 = x^2 a
\end{equation}
for all $x\in Q$. Now,
\[
a(x\cdot ax) = ax\cdot ax = a^2 x^2 = a\cdot ax^2 = a\cdot x^2a = a(x\cdot xa)\,,
\]
using $a\in \lnuc{Q}$ in the first and third equality, endomorphic squaring in the second, \ref{Eqn:AIP_SqEndo-tmp1} in the fourth, and $a\in \rnuc{Q}$ in the fifth. Cancelling $a$ and then $x$ on the left, we get $ax=xa$ for all $x\in Q$, that is, $a\in C(Q)$. This completes the proof of this case, hence the proof of the theorem.
\end{proof}

Since the assumptions of AIP and/or endomorphic squaring might seem rather strong, one might wonder whether Theorem \ref{Thm:AIP_SqEndo} can be improved by assuming that squares lie in just one nucleus. The following examples, all found using \textsc{Mace4}, show that the hypotheses of the theorem are reasonably close to optimal. There are a few unresolved cases we leave as open problems. 

\begin{exm}
Here is a left nuclear square loop with the AIP, but without endomorphic squaring. Here $4^2\cdot 2^2 = 1\cdot 3 = 3$ but $(4\cdot 2)^2 = 6^2 = 2$. Note that in this example, $3^2\not\in C(Q)$ because $3^2\cdot 4 = 2\cdot 4 = 5 \neq 6 = 4\cdot 2 = 4\cdot 3^2$.
\[
\begin{array}{c|cccccc}
\cdot & 1 & 2 & 3 & 4 & 5 & 6 \\
\hline
1     & 1 & 2 & 3 & 4 & 5 & 6 \\
2     & 2 & 3 & 1 & 5 & 6 & 4 \\
3     & 3 & 1 & 2 & 6 & 4 & 5 \\
4     & 4 & 6 & 5 & 1 & 2 & 3 \\
5     & 5 & 4 & 6 & 2 & 3 & 1 \\
6     & 6 & 5 & 4 & 3 & 1 & 2
\end{array}
\]
\end{exm}

\begin{prb}
	Does there exist an AIP loop with left nuclear and commuting squares but without endomorphic squaring?
\end{prb}

\begin{exm}
Here is a middle nuclear square loop with the AIP (hence, by Lemma \ref{Lem:lmrnuc}(4), with commuting squares as well), but without endomorphic squaring. Here $2^2\cdot 4^2 = 3\cdot 1 = 3$, but $(2\cdot 4)^2 = 5^2 = 2$.
\[
\begin{array}{c|cccccc}
	\cdot & 1 & 2 & 3 & 4 & 5 & 6 \\
	\hline
	1     & 1 & 2 & 3 & 4 & 5 & 6 \\
	2     & 2 & 3 & 1 & 5 & 6 & 4 \\
	3     & 3 & 1 & 2 & 6 & 4 & 5 \\
	4     & 4 & 5 & 6 & 1 & 3 & 2 \\
	5     & 5 & 6 & 4 & 3 & 2 & 1 \\
	6     & 6 & 4 & 5 & 2 & 1 & 3
\end{array}
\]
\end{exm}

\begin{prb}
Does there exist a left nuclear square loop with endomorphic squaring but not satisfying the AIP?
\end{prb}

\begin{exm}
Here is a loop with middle nuclear and commuting squares and with endomorphic squaring, but without the AIP. Here $(3\cdot 5)\ldv 1 = 7\ldv 1 = 5$, but $(3\ldv 1)(5\ldv 1) = 3\cdot 8 = 6$.
\[
\begin{array}{c|cccccccc}
	\cdot & 1 & 2 & 3 & 4 & 5 & 6 & 7 & 8\\
	\hline
	1     & 1 & 2 & 3 & 4 & 5 & 6 & 7 & 8 \\
	2     & 2 & 1 & 4 & 3 & 6 & 5 & 8 & 7 \\
	3     & 3 & 4 & 1 & 2 & 7 & 8 & 5 & 6 \\
	4     & 4 & 3 & 2 & 1 & 8 & 7 & 6 & 5 \\
	5     & 5 & 6 & 7 & 8 & 2 & 3 & 4 & 1 \\
	6     & 6 & 5 & 8 & 7 & 3 & 2 & 1 & 4 \\
	7     & 7 & 8 & 5 & 6 & 1 & 4 & 2 & 3 \\
	8     & 8 & 7 & 6 & 5 & 4 & 1 & 3 & 2
\end{array}
\]
\end{exm}

\begin{exm}
Here is a loop with left nuclear squares, endomorphic squaring and the AIP, but without commuting squares. Here $3^3\cdot 3 = 2\cdot 3 = 4 \neq 5 =
3\cdot 2 = 3\cdot 3^2$.
\[
\begin{array}{c|cccccccc}
	\cdot & 1 & 2 & 3 & 4 & 5 & 6 & 7 & 8\\
	\hline
	1 & 1 & 2 & 3 & 4 & 5 & 6 & 7 & 8\\
	2 & 2 & 1 & 4 & 3 & 6 & 5 & 8 & 7\\
	3 & 3 & 5 & 2 & 1 & 7 & 8 & 4 & 6\\
	4 & 4 & 6 & 1 & 2 & 8 & 7 & 3 & 5\\
	5 & 5 & 3 & 7 & 8 & 2 & 1 & 6 & 4\\
	6 & 6 & 4 & 8 & 7 & 1 & 2 & 5 & 3\\
	7 & 7 & 8 & 5 & 6 & 3 & 4 & 1 & 2\\
	8 & 8 & 7 & 6 & 5 & 4 & 3 & 2 & 1
\end{array}
\]
\end{exm}

\begin{prb}
Does there exist a loop with left nuclear and commuting squares, endomorphic squaring and the AIP, but without middle nuclear squares?
\end{prb}

\begin{exm}
Here is a loop with middle nuclear squares, endomorphic squaring and the AIP 
(hence, by Lemma \ref{Lem:lmrnuc}(4), with commuting squares as well) but without left nuclear squares. Here $(3^2\cdot 3)\cdot 3 = (2\cdot 3)\cdot 3 = 4\cdot 3 = 8 \neq 1 = 2\cdot 2 = 3^2\cdot (3\cdot 3)$.
\[
\begin{array}{c|cccccccc}
	\cdot & 1 & 2 & 3 & 4 & 5 & 6 & 7 & 8\\
	\hline
	1 & 1 & 2 & 3 & 4 & 5 & 6 & 7 & 8\\
	2 & 2 & 1 & 4 & 3 & 6 & 5 & 8 & 7\\
	3 & 3 & 4 & 2 & 8 & 1 & 7 & 5 & 6\\
	4 & 4 & 3 & 8 & 2 & 7 & 1 & 6 & 5\\
	5 & 5 & 6 & 1 & 7 & 2 & 8 & 3 & 4\\
	6 & 6 & 5 & 7 & 1 & 8 & 2 & 4 & 3\\
	7 & 7 & 8 & 5 & 6 & 3 & 4 & 1 & 2\\
	8 & 8 & 7 & 6 & 5 & 4 & 3 & 2 & 1
\end{array}
\]
\end{exm}

\section{Decomposition Theorem}
\label{Sec:decomp}

In this section, let $Q$ be a loop in which squaring is a centralizing endomorphism, that is, $Q$ has central squares and the squaring map $s\colon Q\to Q;x\mapsto x^2$ takes its values in $Z(Q)$. We will freely use relevant results of the previous section.

For each nonnegative integer $n$, set
\[
E_n\coloneqq \{ a\in Q \mid a^{2^n} = e \}\quad\text{and}\quad E\coloneqq \bigcup_{n\geq 0} E_n\,. \tag{E}\label{Eqn:E}
\]
Note that each $E_n$ is the kernel of the iterated endomorphism $s^n$. This immediately implies the first two parts of the following.

\begin{lem}\label{Lem:E}
{\ }
\begin{enumerate}
\item Each $E_n$ is a normal subloop of $Q$;
\item $E$ is a normal subloop of $Q$;
\item $Q/E_1$ is an abelian group.
\end{enumerate}
\end{lem}
\begin{proof}
For (3), consider the associator $[x,y,z] = (x\cdot yz)\ldv (xy\cdot z)$ for each $x,y,z\in Q$. Since $s$ is an endomorphism
and squares are central, $[x,y,z]^2 = [x^2,y^2,z^2] = e$. Thus $E_1$ contains every associator and so $Q/E_1$ is a group.
Since $Q/E_1$ also satisfies the AIP, it is an abelian group.
\end{proof}

Note that the assumption that squaring is an endomorphism or, by Theorem \ref{Thm:AIP_SqEndo}, the AIP, is necessary. The dihedral group of order $8$, for instance, has central squares but the elements of order $2$ do not form
a subgroup.

Next, set
\[
O\coloneqq \{ a\in Q\mid a\text{ has finite odd order } \}\,. \tag{O}\label{Eqn:O}
\]
Since every element of $O$ is a square, we have the following.
\begin{lem}\label{Lem:O}
$O$ is a central, hence normal, subloop of $Q$. In particular, $O$ is an abelian group.
\end{lem}

We now have our main decomposition theorem. A loop is said be \emph{torsion} if each $1$-generated subloop is finite. In the power-associative case, this just means that every element has finite order.

\begin{thm}\label{Thm:ExO}
Let $Q$ be a torsion loop in which squaring is a centralizing endomorphism. Define $E$ and $O$ as in \eqref{Eqn:E} and \eqref{Eqn:O}, respectively. Then $Q\cong E\times O$.
\end{thm}
\begin{proof}
Let $a\in Q$ with $a\neq e$ and let $n$ be the order of $a$. Write $n = 2^k m$ where $k\geq 0$ and $m$ is odd.
By B\'{e}zout's identity, there exist integers $i,j$ such that $i\cdot 2^k + j\cdot m = 1$. Set $b\coloneqq a^{j\cdot m}$
and $c\coloneqq a^{i\cdot 2^k}$; thus $bc = a$. Since $b^{2^k} = e$ and $c^m = e$, we have $b\in E$ and $c\in O$.
This proves $Q = EO$. Since each of $E$ and $O$ are normal subloops (Lemmas \ref{Lem:E} and \ref{Lem:O}) and $E\cap O = \{e\}$, we have the desired result.
\end{proof}

\section{Left C loops}
\label{Sec:leftC}

We conclude with a discussion of how the results of this paper specialize to left C loops. Recall that left C loops are defined by any of the equivalent identities mentioned in {\S}\ref{Sec:intro}. But there are useful characterizations \cite{DK}, \cite{Fe}, \cite{PVII}.

\begin{thm}
For a loop $Q$, the following are equivalent:
\begin{enumerate}
	\item $Q$ is a left C loop;
	\item $Q$ has left nuclear squares and the left alternative property;
	\item $Q$ has middle nuclear squares and the left alternative property;
	\item $Q$ has left nuclear squares and the left inverse property;
	\item $Q$ has middle nuclear squares and the left inverse property.
\end{enumerate}
\end{thm}

\noindent In particular, since left C loops have the LIP, the left and middle nuclei coincide (Lemma \ref{Lem:lmrnuc}(4)).

Theorem \ref{Thm:LMNuc_normal} immediately specializes to this setting. 
A proof of the following result's second assertion was first published in \cite{DK}.

\begin{thm}\label{Thm:lnuc_normal}
Let $Q$ be a left C loop and let $N = \lnuc{Q}$. Then $L_{(N)}\lhd \mlt{Q}$ and $N\lhd Q$.
\end{thm}

A \emph{Steiner loop} is a commutative loop satisfying the identity $x\cdot xy = y$. They can be characterized as unipotent C loops, that is, C loops of exponent $2$ ($x^2 = 1$ for all $x$).  Since squares in C loops are nuclear, 
it follows that the quotient of a C loop by its nucleus is a Steiner loop \cite{PV1}. Steiner loops are in one-to-one correspondence with Steiner triple systems, and hence, are important in combinatorics.

For the one-sided version of the preceding discussion, we will use the term \emph{left Steiner loop} to refer to unipotent, left C loops. A loop is left Steiner if and only if it satisfies the identity $x\cdot xy=y$
if and only if it is left alternative and has exponent $2$ if and only if it has the LIP and exponent $2$. The one-sided version of the relationship between a C loop and its quotient Steiner loop is the following.

\begin{prp}\label{Prp:leftSteiner}
The quotient of a left C loop by its left nucleus is a left Steiner loop.
\end{prp}
\begin{proof}
If $Q$ is a left C loop, then since $x^2\in \lnuc{Q}$ for all $x\in Q$, it follows that $Q/\lnuc{Q}$ is a unipotent, left C loop, that is, $Q/\lnuc{Q}$ is left Steiner.
\end{proof}

\emph{Right Steiner loops} are defined and characterized analogously. It is clear from the definitions that a loop is Steiner if and only if it is both left Steiner and right Steiner. Alternatively, this can be seen from a quick calculation: $xy = (xy \cdot x)x = (xy \cdot (xy \cdot y))x = yx$. Thus our suggested terminology is consistent with the general loop theory practice that for a property $\mathcal{P}$, left $\mathcal{P}$ and right $\mathcal{P}$ is equivalent to $\mathcal{P}$. (Bol loops are the obvious exception to this practice: left Bol and right Bol is equivalent to Moufang.)

\begin{cor}
	Every simple left C loop is a group or a left Steiner loop.
\end{cor}
\begin{proof}
If $Q$ is a simple left C loop, then either $L = \nuc{Q}$ or $\nuc{Q} = 1$. In the former case, $Q$ is a simple group. In the latter case, $Q$ is left Steiner by Proposition \ref{Prp:leftSteiner}.
\end{proof}

Finally, between general left C loops and left Steiner loops is the variety of left C loops with central squares. By Theorem \ref{Thm:AIP_SqEndo}, these can also be described as AIP left C loops or as left C loops with endomorphic squaring. In the torsion case, we immediately have the following consequence of Theorem \ref{Thm:ExO}. 

\begin{thm}
Let $Q$ be a torsion, AIP left C loop. Define $E$ and $O$ as in \eqref{Eqn:E} and \eqref{Eqn:O}, respectively. Then $Q\cong E\times O$.
\end{thm}

We conclude with an aside: the variety of AIP left C loops can be characterized by a single identity; we omit the easy proof.

\begin{prp}
The variety of AIP left C loops is axiomatized, in the variety of loops, by the identity $x \cdot (y \cdot yx)z = yx\cdot (yx \cdot z)$, that is, 
$L_x L_{y\cdot yx} = L_{yx}^2$.
\end{prp}

\section*{Acknowledgments}

This work was supported by the automated theorem prover \textsc{Prover9} and the finite model builder \textsc{Mace4}, both created by McCune \cite{McCune}.

\end{document}